\documentclass[reqno,a4paper,10pt]{amsart}

\usepackage[pdfpagelabels]{hyperref}
\usepackage{amssymb}
\usepackage{url}
\usepackage{color}

\newtheorem{theorem}{Theorem}

\newtheorem{corollary}[theorem]{Corollary}
\newtheorem{proposition}[theorem]{Proposition}

\theoremstyle{definition}

\theoremstyle{remark}

\newtheorem{example}{Example}

\newcommand{\abs}[1]{\lvert#1\rvert}
\newcommand{\nm}[1]{\lVert#1\rVert}

\newcommand{\D}{\mathbb{D}}

\newcommand{\N}{\mathbb{N}}

\newcommand{\C}{\mathbb{C}}

\renewcommand{\phi}{\varphi}

\newcommand{\BMOA}{\rm BMOA}

\renewcommand{\H}{\mathcal{H}}

\begin{document}

\title[Factorization of solutions of linear differential equations]{Factorization of solutions of linear differential equations}

\author{Janne Gr\"ohn}
\address{Department of Physics and Mathematics, University of Eastern Finland\newline
\indent P.O. Box 111, FI-80101 Joensuu, Finland}
\email{janne.grohn@uef.fi}

\date{\today}


\begin{abstract}
  This paper supplements recents results on linear differential equations $f''+Af=0$, where the coefficient
  $A$ is analytic in the unit disc of the complex plane~$\C$. It is shown that, if $A$ is analytic and $|A(z)|^2(1-|z|^2)^3\, dm(z)$ is
  a~Carleson measure, then all non-trivial solutions of $f''+Af=0$ can be factorized as $f=Be^g$,
  where $B$ is a~Blaschke product whose zero-sequence $\Lambda$ is uniformly separated and where
  $g\in{\rm BMOA}$ satisfies the interpolation property
  $$g'(z_n) = -\frac{1}{2} \, \frac{B''(z_n)}{B'(z_n)}, \quad z_n\in\Lambda.$$
  Among other things, this factorization
  implies that all solutions of $f''+Af=0$ are functions in a~Hardy space and have no singular inner factors.

  Zero-free solutions play an~important role as their maximal growth is similar to the general case.  
  The study of zero-free solutions produces a~new result on Riccati differential equations.
\end{abstract}

\dedicatory{Dedicated to the memory of Christian Pommerenke}
\maketitle


\section{Introduction} 

Let $\H(\D)$ be the collection of analytic functions in the unit disc~$\D = \{ z\in\C : |z|<1\}$.
This study concerns the growth of solutions of
\begin{equation} \label{eq:de2}
f''+Af=0,
\end{equation}
where the coefficient $A\in\H(\D)$. Our aim is to find a~new factorization
for solutions of~\eqref{eq:de2} under an appropriate coefficient condition.
Such factorization allows us to improve and extend many recent results. More specific details
on improvements are given in Section~\ref{sec:results}.

To consider the growth of solutions $f$ of~\eqref{eq:de2}, we recall the following definitions.
For $0<p<\infty$
the Hardy space $H^p$  consists of those
functions in $\H(\D)$ for which
 \begin{equation*}
 \nm{f}_{H^p} = \lim_{r\to 1^-} \left( \frac{1}{2\pi} \int_0^{2\pi}
 \abs{f(r e^{i\theta})}^p \, d\theta \right)^{1/p}<\infty.
 \end{equation*}
A positive Borel measure $\mu$ on $\D$ is called a Carleson measure, if there exists a~positive constant 
$C$ such that
\begin{equation*}
\int_{\D} |f(z)|^p \, d\mu(z) \leq C \, \nm{f}_{H^p}^p, \quad f\in H^p;
\end{equation*}
or equivalently, if
\begin{equation*} 
\sup_{a\in\D} \, \int_\D |\varphi_a'(z)| \, d\mu(z) < \infty.
\end{equation*}
Here $\varphi_a(z) = (a-z)/(1-\overline{a} z)$ is a~conformal automorphism of $\D$
which coincides with its own inverse.
For more information on Carleson measures, we refer to \cite{D:2000}.
The Bloch space $\mathcal{B}$ consists of functions in $\H(\D)$ for which  $$\|f\|_{\mathcal{B}} = \sup_{z\in\D} \, |f'(z)|(1-|z|^2)<\infty,$$ while
the space $\BMOA$ consists of those functions in $H^2$
whose boundary values have bounded mean oscillation on $\partial\D$; or equivalently, of those analytic functions $f\in\H(\D)$
for which $|f'(z)|^2(1-|z|^2)\, dm(z)$ is a~Carleson measure. Here $d m(z)$ denotes the element of the Lebesgue area measure on~$\D$.
Denote
$$\nm{f}_{\rm BMOA}^2 = \sup_{a\in\D} \, \nm{f_a}_{H^2}^2,$$
where $f_a(z)=f(\varphi_a(z)) - f(a)$ for all $a,z\in\D$.

To consider the oscillation of solutions $f$ of~\eqref{eq:de2}, we recall the following definitions.
Let
$\varrho(z,w)=|z-w|/|1-\overline{z}{w}|$ denote the pseudo-hyperbolic distance between
the points $z,w\in\D$. The pseudo-hyperbolic disc of radius $0<\delta<1$, centered
at $z\in\D$, is given by $\Delta(z,\delta) = \{ w \in\D : \varrho(z,w) < \delta \}$.
The sequence $\Lambda\subset \D$ is called separated in the hyperbolic metric if there exists $0<\delta<1$ such that
$\varrho(z_n,z_k) >\delta$ for all $z_n,z_k\in\Lambda$ for which $z_n\neq z_k$.
Sequence $\Lambda$ is said to be uniformly separated if it is separated in the hyperbolic metric
and $\sum_{z_n\in\Lambda} (1-|z_n|)\delta_{z_n}$ is a~Carleson measure. Here $\delta_{z_n}$ is the Dirac measure with point mass at~$z_n$. 
If $B$ is a~Blaschke product whose zero sequence is $\Lambda$, then
$\Lambda$ is uniformly separated if and only if there
exists a~constant $C=C(B)>0$ such that
\begin{equation*} 
|B(z)| \geq C \varrho(z,\Lambda), \quad z\in\D;
\end{equation*}
see for example \cite[p.~217]{N:1986}.
Uniformly separated sequences $\Lambda$
satisfy the Blaschke condition $\sum_{z_n\in\Lambda} (1-|z_n|) <\infty$, while
there are separated sequences for which the Blaschke condition fails to be true \cite[pp.~214--215]{S:1993}.

To consider the growth of the coefficient $A$, we recall the following definitions.
The growth space $H^\infty_\alpha$ consists of those functions
in $\H(\D)$ for which
\begin{equation*}
\|A\|_{H^\infty_\alpha} = \sup_{z\in\D} \,  \abs{A(z)} (1-\abs{z}^2)^\alpha < \infty.
\end{equation*} 
We are especially interested in analytic coefficients for which
$|A(z)|^2(1-|z|^2)^3\, dm(z)$ is a~Carleson measure. Then
$A\in H^\infty_2$ by subharmonicity.


\section{Results} \label{sec:results}

In Section~\ref{subsec:main} we present our main results on solutions of~\eqref{eq:de2},
where $A\in\H(\D)$ and $|A(z)|^2(1-|z|^2)^3\, dm(z)$ is a~Carleson measure. The proofs of the main results
depend on intermediate results, which are considered in Section~\ref{subsec:inter} separately.
Zero-free solutions are discussed in detail in Section~\ref{subsec:zerofree}, which leads
us to prove a~related result on Riccati differential equations.


\subsection{The main results} \label{subsec:main}

If $A\in\H(\D)$ and $|A(z)|^2(1-|z|^2)^3\, dm(z)$ is a~Carleson measure, then
zero-sequences of all non-trivial solutions ($f\not\equiv 0$) of \eqref{eq:de2} are uniformly separated \cite[Corollary~3]{GN:2017}.
Conversely, if $\Lambda$ is any uniformly separated sequence, then there exists $A\in\H(\D)$ for which $|A(z)|^2(1-|z|^2)^3\, dm(z)$ is a~Carleson measure and for which~\eqref{eq:de2} admits a~non-trivial solution whose zero sequence is $\Lambda$ \cite[Corollary~2]{G:2017}.
These results offer a~complete description of zero-sequences under this coefficient condition.

If $A\in\H(\D)$ and $|A(z)|^2(1-|z|^2)^3\, dm(z)$ is a~Carleson measure, then
zero-free solutions $f$ of~\eqref{eq:de2} are of the form $f=e^g$, where $g\in {\rm BMOA}$. Conversely, for any $g\in{\rm BMOA}$ the function $f=e^g$ is a~solution of~\eqref{eq:de2} for some $A\in\H(\D)$ for which $|A(z)|^2(1-|z|^2)^3\, dm(z)$ is a~Carleson measure. See~\cite[Theorem~4(i)]{GNR:2018} for both of these assertions. These results offer a~complete description of zero-free solutions
under this coefficient condition.

The following factorization completes and unites the previous findings mentioned above.


\begin{theorem} \label{thm:factorization_main}
  If $A\in\H(\D)$ and $|A(z)|^2(1-|z|^2)^3\, dm(z)$ is a~Carleson measure, then all non-trivial solutions $f$ of~\eqref{eq:de2} can be
  factorized as $f=Be^g$, where $B$ is a~Blaschke
product whose zero-sequence $\Lambda$ is uniformly separated and where $g\in{\rm BMOA}$ satisfies the interpolation property
\begin{equation} \label{eq:iprop}
  g'(z_n) = -\frac{1}{2} \, \frac{B''(z_n)}{B'(z_n)}, \quad z_n\in\Lambda.
\end{equation}
    
  Conversely, if $B$ is any Blaschke product whose zero-sequence $\Lambda$ is uniformly separated
   and $g\in {\rm BMOA}$ is any function which satisfies the interpolation property~\eqref{eq:iprop}, then $f=Be^g$ is
  a~solution of~\eqref{eq:de2} for some $A\in\H(\D)$ for which $|A(z)|^2(1-|z|^2)^3\, dm(z)$ is a~Carleson measure.
\end{theorem}

The factorization $f=Be^g$ ensured by Theorem~\ref{thm:factorization_main} contains no singular inner factors,
which follows from \cite[Corollary~3, p.~34]{D:2000} as $\log\frac{f}{B} = g \in {\rm BMOA} \subset H^1$.

The class of normal functions consists of those meromorphic functions $f$ in $\D$ for which
$\sup_{z\in\D} \, f^{\#}(z) (1-|z|^2) < \infty$, where $f^\# = |f'|/(1+|f|^2)$ is the spherical derivative;
equivalently, $f$ is normal if and only if
$\{ f \circ \phi : \text{$\phi$ conformal automorphism of $\D$}\}$ is a~normal family in 
the sense of Montel \cite{LV:1957}.
If $A\in\H(\D)$ and $|A(z)|^2(1-|z|^2)^3\, dm(z)$ is a~Carleson measure, then
all normal solutions of~\eqref{eq:de2} belong to $\bigcup_{0<p<\infty} H^p$ by \cite[Corollary~9]{GNR:2018},
whose proof relies on a~highly involved stopping time argument.
This result is incomplete as non-normal solutions are possible; see the proof of \cite[Theorem~3]{G:2017_1} where
$|A(z)|^2(1-|z|^2)^3\, dm(z)$ is in fact a~Carleson measure.
The factorization ensured by Theorem~\ref{thm:factorization_main} allows us to extend \cite[Corollary~9]{GNR:2018} to all solutions.


\begin{corollary} \label{cor:factorization_main}
  If $A\in\H(\D)$ and $|A(z)|^2(1-|z|^2)^3\, dm(z)$ is a~Carleson measure, then all solutions of~\eqref{eq:de2} belong to $\bigcup_{0<p<\infty} H^p$.
\end{corollary}

Since $e^g \in \bigcup_{0<p<\infty} H^p$ for any $g\in{\rm BMOA}$ by \cite[Theorem~1]{CS:1976}, Corollary~\ref{cor:factorization_main}
is an~immediate consequence of Theorem~\ref{thm:factorization_main} and needs no further proof.


\subsection{The intermediate case} \label{subsec:inter}
If $A\in H^\infty_2$, then all solutions of~\eqref{eq:de2} belong to the Korenblum space $\bigcup_{0<\alpha<\infty} H^\infty_\alpha$.
This fact can be proved by classical comparison theorem \cite[Example~1]{P:1982};
growth estimates \cite[Theorem 4.3(2)]{H:2000}, \cite[Theorem~3.1]{HKR:2007}; successive approximations \cite[Theorem~I]{G:2011};
and straight-forward integration \cite[Theorem~2]{GR:2017}, \cite[Corollary~4(a)]{HKR:2016}.
Some solutions of~\eqref{eq:de2} may lie outside $\bigcup_{0<p<\infty} H^p$.
This can happen for solutions with no zeros \cite[Theorem~3]{G:2018} as well as for solutions
with too many zeros \cite[Corollary~1]{G:2017}, where the Blaschke condition fails to be true.

If $A\in H^\infty_2$, then zero-sequences of all non-trivial solutions of \eqref{eq:de2} are separated, and vice versa
\cite[Theorems~3 and~4]{S:1955}.
If $\Lambda$ is a~separated sequence of sufficiently small upper uniform
density, then there exists $A\in H^\infty_2$ such that~\eqref{eq:de2} admits a~non-trivial solution that vanishes at all points $\Lambda$
\cite[Theorem~1]{G:2017}, but has also other zeros.

The results above show that much is known about the behaviour of solutions of~\eqref{eq:de2} in the case  $A\in H^\infty_2$.
However, we are far away from the complete description of solutions as in Section~\ref{subsec:main}. There
is one exception. If $A\in H^\infty_2$, then
all zero-free solutions $f$ of~\eqref{eq:de2} are of the form $f=e^g$, where $g\in\mathcal{B}$. Conversely, for any $g\in\mathcal{B}$ the function $f=e^g$ is a~solution of~\eqref{eq:de2} for some $A\in H^\infty_2$. See~\cite[Theorem~4(ii)]{GNR:2018} for both of these
assertions, giving a~complete description of zero-free solutions in the case $A\in H^\infty_2$.

We proceed to consider an~intermediate case, where we assume that the solution has uniformly separated zeros
under the coefficient condition $A\in H^\infty_2$. This allows us to prove a~factorization similar to Theorem~\ref{thm:factorization_main}.


\begin{theorem} \label{thm:factorization_h2}
  If $A\in H^\infty_2$ and $f$ is a~solution of~\eqref{eq:de2} whose zero-sequence $\Lambda$ is uniformly separated, then
  $f$ can be factorized as $f=Be^g$, where $B$ is a~Blaschke
  product whose zero-sequence is $\Lambda$ and where $g\in\mathcal{B}$ satisfies the interpolation property~\eqref{eq:iprop}.

  Conversely, if $B$ is any Blaschke product whose zero-sequence $\Lambda$ is uniformly separated
   and $g\in\mathcal{B}$ satisfies the interpolation property~\eqref{eq:iprop}, then $f=Be^g$ is
   a~solution of~\eqref{eq:de2} for some $A\in H^\infty_2$.
 \end{theorem}

Even in this intermediate case some solutions may lie outside of $\bigcup_{0<p<\infty} H^p$ by \cite[Theorem~3]{G:2018}.
However, solutions of~\eqref{eq:de2} are close to $H^p$ as $p\to 0^+$ in the following sense.
Corollary~\ref{cor:factorization_h2} follows as \cite[Lemma~5.3]{P:1975} is applied to the function $e^g$
in the factorization $f=Be^g$, and it needs no further proof.


\begin{corollary} \label{cor:factorization_h2}
If $A\in H^\infty_2$ and $f$ is a~solution of~\eqref{eq:de2} whose zero-sequence $\Lambda$ is uniformly separated,
then $f=Be^g$ as in Theorem~\eqref{thm:factorization_h2} and
\begin{equation*} 
\sup_{0<r<1}  \left( \frac{1}{2\pi}\int_0^{2\pi} |f(re^{i\theta})|^p \, d\theta  \right)^{1/p} (1-r)^{p\, \|g\|_{\mathcal{B}}^2} < \infty
\end{equation*}
for any $0<p<\infty$.
\end{corollary}

The following result concerns local behavior of non-trivial solutions around their zeros.


\begin{theorem} \label{thm:bloch}
  Suppose that $A\in H^\infty_2$ and $f$ is a~solution of~\eqref{eq:de2} whose zero-sequence $\Lambda$ is uniformly separated.
    Define $h_{z_n}(z) = f(\varphi_{z_n}(\delta z))$ for $z_n\in\Lambda$ and $0<\delta<1$. Then
$\| h_{z_n}''/h_{z_n}' \|_{H^\infty} \lesssim \delta$
uniformly for all $z_n\in\Lambda$.
\end{theorem}

In the case of Theorem~\ref{thm:bloch}, the restriction $f |_{\Delta(z_n,\delta)}$
of a~non-trivial solution~$f$
to a~pseudo-hyperbolic disc around its zero $z_n\in\Lambda$, is a~univalent function.
See the discussion after the proof of Theorem~\ref{thm:bloch} for more details.


\subsection{Zero-free solutions} \label{subsec:zerofree}

Theorems~\ref{thm:factorization_main}~and~\ref{thm:factorization_h2} imply that
the maximal growth of solutions of~\eqref{eq:de2} is similar to the maximal growth
of zero-free solutions described in~\cite[Theorem~4]{GNR:2018}.
We proceed to consider zero-free solutions of~\eqref{eq:de2} in a~more general setting.
For $X\subset \mathcal{H}(\D)$ and $n\in\N$, let $X^{(n)} = \{ f^{(n)} : f\in X \}$
denote the collection of $n$th derivatives of functions in $X$. We apply the classical notation
$X'$ and $X''$ in the case of first and second derivatives, respectively.
For example, $\mathcal{B}'' = H^\infty_2$ by standard estimates,
while ${\rm BMOA}''$ consists of those functions $f\in\H(\D)$ for which $|f''(z)|^2(1-|z|^2)^3\, dm(z)$
is a~Carleson measure by \cite[Theorem~3.2]{R:2003}.

For any $g\in\H(\D)$, define the non-linear operator
\begin{equation*} 
  R(g)(z) = \int_0^z \!\!\int_0^\zeta \big( - g''(\xi) - g'(\xi)^2 \big) \, d\xi\, d\zeta, \quad z\in\D.
\end{equation*}
In this paper, the~linear subspace $X\subset \H(\D)$ is said to be \emph{admissible} if the following conditions hold:
\begin{enumerate}
\item[\rm (i)] $X$ contains all polynomials (restricted to $\D$);
\item[\rm (ii)] for all $g\in\H(\D)$, we have $g\in X$ if and only if $R(g)\in X$.
\end{enumerate}
Note that either of the terms $g''$ and $(g')^2$ in the definition of $R(g)$ may be dominant. For example, if $g'$ is an~infinite Blachke
product, then $(g')^2$ is bounded while $g''$ is unbounded. Conversely, if $g'(z)= (1-z)^{-2}$, then $g'(z)^2 = (1-z)^{-4}$
grows a~lot faster than $g''(z) = 2(1-z)^{-3}$.
    
Many classical function spaces are admissible, see
Section~\ref{sec:examples}. The following result describes admissibility in terms of Schwarzian derivatives.
For meromorphic functions $w$ in $\D$, the pre-Schwarzian derivative $P_w$ and
the Schwarzian derivative $S_w$ are defined as
\begin{equation*}
P_w = \frac{w''}{w'}\, , \quad S_w = P_w ' - \frac{1}{2} \left( P_w \right)^2.
\end{equation*}


\begin{proposition} \label{prop:cha}
  Let $X\subset \H(\D)$ be a~linear subspace, which contains all polynomials (restricted to $\D$).
  The following are equivalent:
  \begin{enumerate}
  \item[\rm (i)] $X$ is admissible;
  \item[\rm (ii)] for all locally univalent $w\in\H(\D)$,
    \begin{equation*}
  P_w\in X' \quad \Longleftrightarrow \quad S_w \in X''.
\end{equation*}
  \end{enumerate}
\end{proposition}

The following result is a~generalization of \cite[Theorem~4]{GNR:2018}, which
corresponds to the cases $X=\mathcal{B}$ and $X={\rm BMOA}$.
The philosophy behind Proposition~\ref{prop:aux} originates from the following identity.
If $f=e^g$ is a non-vanishing solution of \eqref{eq:de2}, then $g$ is a~solution of
the Riccati differential equation $g'' = -A - (g')^2$.


\begin{proposition} \label{prop:aux}
Let $X$ be an~admissible linear subspace of $\H(\D)$, and
let $f$ be a~zero-free solution of~\eqref{eq:de2}.
Then $A\in X''$ if and only if $f=e^g$ where $g \in X$.
\end{proposition}

As the final objective, we study
Riccati differential equations.
If $A\in H^\infty_2$, $B\in H^\infty_1$ and $C\in H^\infty_0$ with
$\inf_{z\in\D} |C(z)|>0$, then each analytic solution $g$ of
\begin{equation} \label{eq:riccati}
g'' = A +B \, g' + C \, (g')^2 
\end{equation}
satisfies $g\in\mathcal{B}$ by \cite[Theorem~1]{Y:1977}.
This results fails to be true without the assumption $\inf_{z\in\D} |C(z)|>0$, since 
$g(z)=1/(1-z)\not\in\mathcal{B}$ is a~solution of~\eqref{eq:riccati} for $A\equiv 0$, $C\equiv 0$ and $B(z)=2/(1-z)$,
 $z\in\D$.
The following theorem generalizes \cite[Theorem~1]{Y:1977} in two ways.
The coefficient conditions are given in terms of admissible linear subspaces, and in certain special cases,
it extends to the case  $\inf_{z\in\D} |C(z)|=0$.


\begin{theorem} \label{thm:riccati}
Let $X$ be an~admissible linear subspace of $\mathcal{H}(\D)$.
If $A,B$ are meromorphic functions in $\D$ and $C\in\mathcal{H}(\D)$ such that $AC\in X''$ and $(B+C'/C)\in X'$,
then each analytic solution $g$ of~\eqref{eq:riccati} satisfies $Cg'\in X'$.
\end{theorem}

Although the coefficients $A$ and $B$ are allowed to be meromorphic in Theorem~\ref{thm:riccati},
the auxiliary functions $AC$, $B+C'/C$ and $C$ are required to be analytic in $\D$.


\begin{example}
Let $C(z) = 1-z$ for $z\in\D$. Then $\inf_{z\in\D} |C(z)|=0$ and $C'(z)/C(z) = -1/(1-z)$, $z\in\D$. 
Now Theorem~\ref{thm:riccati} can be applied for $X=\mathcal{B}$,
$A\in H^\infty_2$ and $B\in H^\infty_1$.
\end{example}


\begin{example}
Let $C(z) = z$ for $z\in\D$. Then~$\inf_{z\in\D} |C(z)|=0$ trivially.
In this case Theorem~\ref{thm:riccati} can be applied for $X=\mathcal{B}$, $A\in H^\infty_2$
and $B = B^\star- 1/z$, where $B^\star\in H^\infty_1$.
\end{example}


\section{Proofs of the results}

The proof of Theorem~\ref{thm:factorization_main} depends on the proof of Theorem~\ref{thm:factorization_h2}, which
is presented first.


\begin{proof}[Proof of Theorem~\ref{thm:factorization_h2}]
  Let $A\in H^\infty_2$ and let $f$ be a~solution of~\eqref{eq:de2} whose zero-sequence $\Lambda$ is uniformly separated.
  Let $B$ be the Blaschke product whose zero-sequence is $\Lambda$. Then $f/B$ is a~non-vanishing
  analytic function in $\D$, and therefore there exists $g\in\H(\D)$ for which $f=Be^g$. The representation $A=-f''/f \in\H(\D)$
  implies that $f''(z_n)=0$ for all $z_n\in\Lambda$, and therefore~\eqref{eq:iprop} holds
  by a~straight-forward computation. It remains to prove that $g\in\mathcal{B}$.

  Let $f^\star$ be another solution of~\eqref{eq:de2}
  linearly independent to $f$, and define the meromorphic function $w$ in $\D$ as $w=f^\star/f$. It follows that
 \begin{equation*}
S_w = \left( \frac{w''}{w'} \right)' - \frac{1}{2} \left(\frac{w''}{w'} \right)^2 = 2A,
\end{equation*}
and therefore $S_w\in H^\infty_2$. It is well-known that this implies $w$ to be uniformly locally univalent;
see for example \cite[Lemma~B]{G:2017_1}.
This means that $w$~is univalent in every pseudo-hyperbolic disc $\Delta(z,\eta)$ for $z\in\D$, if $0<\eta<1$ is chosen to be
\begin{equation*}
\eta = \min\left\{ 1, \sqrt{2}\, \| S_w\|_{H^\infty_2}^{-1/2}  \, \right\}.
\end{equation*}

Since $\Lambda$ is uniformly separated, there exists a~constant $0<\delta <1$ such that
the pseudo-hyperbolic discs $\Delta(z_n,\delta)$ for $z_n\in\Lambda$ are pairwise disjoint.
Let
$$
\Omega = \big\{ z\in\D : \varrho(z,\Lambda)\geq \delta \big\}.
$$
Now~\cite[Lemma~11]{GNR:2018} implies
\begin{equation*}
  2 \, \left| \frac{f'(z)}{f(z)} \right| (1-|z|^2) =   \left| \frac{w''(z)}{w'(z)} \right| (1-|z|^2)  \leq \frac{6}{\min\{\eta,\delta\}},
  \quad z\in\Omega.
\end{equation*}
By writing $f=Be^g$, the previous inequality is equivalent to
\begin{equation*}
\left| \frac{B'(z) + B(z)g'(z)}{B(z)} \right| (1-|z|^2)  \leq \frac{3}{\min\{\eta,\delta\}}, \quad z\in\Omega,
\end{equation*}
which reduces to
\begin{equation*}
  |g'(z)|(1-|z|^2) \leq \frac{3}{\min\{\eta,\delta\}} + \left| \frac{B'(z)}{B(z)} \right| (1-|z|^2)
  \leq \frac{3}{\min\{\eta,\delta\}} + \frac{\|B\|_{\mathcal{B}}}{\displaystyle\inf_{z\in\Omega} |B(z)|}, \quad z\in\Omega.
\end{equation*}
Since the Blaschke product $B$ satisfies the weak embedding property
$\inf_{z\in\Omega} |B(z)|>0$ \cite[Lemmas~1 and 3]{K-L:1969},
we conclude
\begin{equation*}
\sup \bigg\{ |g'(z)| (1-|z|^2) : z \in \D \setminus \bigcup_n \, \Delta(z_n,\delta) \bigg\} < \infty.
\end{equation*}
Finally \cite[Lemma~1]{G:2018} gives $g\in\mathcal{B}$, which completes the proof of the first part.

The second part of the proof follows from the representation
\begin{equation} \label{eq:crep}
-A = \frac{f''}{f} = \frac{B''+2B'g'}{B} + (g')^2 + g''.
\end{equation}
Since $g\in\mathcal{B}$, we have $(g')^2+g''\in H^\infty_2$. Since $B$ is uniformly separated,
the weak embedding property implies $\inf_{z\in\Omega} |B(z)|>0$ as above, and therefore
\begin{equation} \label{eq:dest}
\left| \frac{B''(z)+2B'(z)g'(z)}{B(z)} \right| (1-|z|^2)^2
\end{equation}
is uniformly bounded for $z\in\Omega$ by standard estimates.
Note that $|B(z)|\gtrsim \varrho(z,z_n)$ for all $z\in \Delta(z_n,\delta)$ by \cite[Lemmas~1 and 3]{K-L:1969}.
An application of \cite[Lemma~1]{G:2017} to $B''+2B'g'\in H^\infty_2$,
which vanishes at all points $z_n\in\Lambda$ by the interpolation property~\eqref{eq:iprop}, reveals that
\eqref{eq:dest} is uniformly bounded also for $z\in\bigcup_{z_n\in\Lambda} \Delta(z_n,\delta)$. Finally,
we conclude that $A\in H^\infty_2$.
\end{proof}

Now we are in a~position to prove our main result.


\begin{proof}[Proof of Theorem~\ref{thm:factorization_main}]
  Let $f$ be a~non-trivial solution of~\eqref{eq:de2}, where the coefficient $A\in\H(\D)$ and $|A(z)|^2(1-|z|^2)^3\, dm(z)$ is a Carleson measure.
  According to~\cite[Corollary~3]{GN:2017}, the zero-sequence $\Lambda$
  of $f$ is uniformly separated. Theorem~\ref{thm:factorization_h2}, whose proof is presented above, implies that
  $f=Be^g$, where $B$ is a~Blaschke product whose zero-sequence $\Lambda$ is uniformly separated
  and $g\in\mathcal{B}$ satisfies the interpolation property~\eqref{eq:iprop}. It remains to prove that $g\in{\rm BMOA}$.
  We apply a~result \cite[Corollary~7]{BJ:1994} due to Bishop and Jones to the identity
  \begin{equation} \label{eq:expand}
    g'' + (g')^2 = - A  - \frac{B''+2B'g'}{B},
  \end{equation}
  which follows from~\eqref{eq:de2} by writing $f=Be^g$.

  Define $g_a(z) = g(\varphi_a(z)) - g(a)$ for all $a,z\in\D$.
  By applying \cite[Corollary~7]{BJ:1994} to $-2g_a$, there exists a~constant $C>0$ such that
\begin{align*}
  \|-2 g_a\|^2_{H^2} & \leq C \, |-2 g_a'(0)|^2 + C \int_{\D} \Big| (-2 g_a)''(z)-\frac{1}{2} \, \big((-2 g_a)'(z)\big)^2 \Big|^2 (1-\abs{z}^2)^3 \, dm(z)\\
  & \leq 4C \, \abs{g_a'(0)}^2 + 4 C \int_{\D} \left| g_a''(z)+g_a'(z)^2 \right|^2 (1-\abs{z}^2)^3 \, dm(z).
\end{align*}
We compute
\begin{align*}
  g_a'(z)  = g'(\varphi_a(z))\varphi'_a(z), \qquad
  g_a''(z)  = g''(\varphi_a(z)) \big( \varphi_a'(z) \big)^2 + g'(\varphi_a(z)) \varphi_a''(z),
\end{align*}
and therefore
\begin{align}
  \sup_{a\in\D} \|g_a\|^2_{H^2} 
  & \leq C \, \sup_{a\in\D} \, \abs{g'(a)}^2(1-|a|^2)^2 \label{eq:b} \\
  & \qquad +  2C \, \sup_{a\in\D} \, \int_{\D} \left| g''(\varphi_a(z)) +g'(\varphi_a(z))^2 \right|^2 \big| \varphi_a'(z) \big|^4  (1-\abs{z}^2)^3 \, dm(z)
    \label{eq:bb}\\
  & \qquad +  2C \, \sup_{a\in\D} \, \int_{\D} \left| g'(\varphi_a(z)) \varphi_a''(z) \right|^2 (1-\abs{z}^2)^3 \, dm(z).\label{eq:bbb} 
\end{align}
The right-hand side of~\eqref{eq:b} is bounded above by a~constant multiple of $\|g\|_{\mathcal{B}}^2<\infty$.
The expression \eqref{eq:bbb} is bounded above by a~constant multiple of
\begin{equation*}
\|g\|_{\mathcal{B}}^2 \, \sup_{a\in\D} \, \int_{\D} \left| \frac{\varphi_a''(z))}{\varphi_a'(z) }\right|^2 (1-\abs{z}^2) \, dm(z)<\infty,
\end{equation*}
which follows from standard estimates. It remains to estimate
\eqref{eq:bb}, which reduces to 
\begin{equation} \label{eq:reduced}
  2C \, \sup_{a\in\D} \, \int_{\D} \left| g''(z) +g'(z)^2 \right|^2 (1-\abs{z}^2)^2 (1-|\varphi_a(z)|^2) \, dm(z)
\end{equation}
after a~conformal change of variables. By applying~\eqref{eq:expand}, we see that \eqref{eq:reduced} is bounded above by
a~constant multiple of
\begin{align}
  & \sup_{a\in\D} \, \int_{\D} \left| A(z) \right|^2 (1-\abs{z}^2)^2 (1-|\varphi_a(z)|^2) \, dm(z) \label{eq:1}\\
  & \qquad + \sup_{a\in\D} \, \int_{\D} \left| \frac{B''(z)+2B'(z)g'(z)}{B(z)} \right|^2 (1-\abs{z}^2)^2 (1-|\varphi_a(z)|^2) \, dm(z) \label{eq:2}.
\end{align}
The supremum in~\eqref{eq:1} is finite, since $|A(z)|^2(1-|z|^2)^3\, dm(z)$ is a Carleson measure. It suffices to
consider~\eqref{eq:2}.
Let $0<\delta<1$ be a sufficiently small constant such that the pseudo-hyperbolic discs
$\Delta(z_n,\delta)$ around zeros $z_n\in\Lambda$ are pairwise disjoint.
Denote $$\Omega = \, \D \setminus \, \bigcup_{z_n\in\Lambda} \Delta(z_n,\delta).$$
On one hand, $|B|$ is uniformly bounded away from zero in $\Omega$ by \cite[Lemmas~1 and 3]{K-L:1969} and $g\in\mathcal{B}$, 
and therefore
\begin{align*} 
  & \sup_{a\in\D} \, \int_{\Omega} \left| \frac{B''(z)+2B'(z)g'(z)}{B(z)} \right|^2 (1-\abs{z}^2)^2 (1-|\varphi_a(z)|^2) \, dm(z)\\
  &  \qquad \lesssim \sup_{a\in\D} \, \int_{\D} | B''(z) |^2 (1-\abs{z}^2)^2 (1-|\varphi_a(z)|^2) \, dm(z)\\
  &  \qquad \qquad + 2 \, \| g \|_{\mathcal{B}}^2 \, \sup_{a\in\D} \, \int_{\D} | B'(z) |^2  (1-|\varphi_a(z)|^2) \, dm(z),
\end{align*}
where both supremums are finite since $|B''(z)|^2(1-|z|^2)^3\, dm(z)$ and $|B'(z)|^2(1-|z|^2)\, dm(z)$
are Carleson measures as the Blaschke product $B$ is bounded; see \cite[Theorem~3.2]{R:2003}.

On the other hand, we write
\begin{align*}
  & \sup_{a\in\D} \, \int_{\bigcup_{z_n\in\Lambda} \!\Delta(z_n,\delta)} \left| \frac{B''(z)+2B'(z)g'(z)}{B(z)} \right|^2 (1-\abs{z}^2)^2 (1-|\varphi_a(z)|^2) \, dm(z)\\
  &  \qquad = \sup_{a\in\D} \, \sum_{z_n\in\Lambda} \int_{\Delta(z_n,\delta)} \left| \frac{B''(z)+2B'(z)g'(z)}{B(z)} \right|^2 \, \frac{(1-|z|^2)^3 (1-|a|^2)}{|1-\overline{a}z|^2} \, dm(z).
\end{align*}
Note that $|B(z)|\gtrsim \varrho(z,z_n)$ for all $z\in \Delta(z_n,\delta)$ by \cite[Lemmas~1 and 3]{K-L:1969},
and by standard estimates $|1-\overline{a}z| \simeq |1-\overline{a}z_n|$
for all $z\in \Delta(z_n,\delta)$ and for all $a\in\D$, with comparison constants
independent of $a$. By applying \cite[Lemma~1]{G:2017}, we obtain
\begin{align*}
  & \sup_{a\in\D} \, \int_{\bigcup_{z_n\in\Lambda} \!\Delta(z_n,\delta)} \left| \frac{B''(z)+2B'(z)g'(z)}{B(z)} \right|^2 (1-\abs{z}^2)^2 (1-|\varphi_a(z)|^2) \, dm(z)\\
  &  \qquad \lesssim \, \sup_{a\in\D} \, \sum_{z_n\in\Lambda} \frac{(1-|z_n|^2)(1-\abs{a}^2)}{|1-\overline{a}z_n|^2} < \infty,
\end{align*}
since $\Lambda$ is uniformly separated. This 
completes the first part of the proof.

The second part of the proof follows from \eqref{eq:crep} by applying
estimates reminiscent to those above. We leave the details to the interested reader.
\end{proof}


\begin{proof}[Proof of Theorem~\ref{thm:bloch}]
  Suppose that $A\in H^\infty_2$ and $f$ is a~solution of~\eqref{eq:de2} whose zero-sequence~$\Lambda$ is uniformly separated.
  Define $h_{z_n}(z) = f(\varphi_{z_n}(\delta z))$ for $z_n\in\Lambda$ and $0<\delta<1$.
 
   For all sufficiently small $\delta>0$,
  the pseudo-hyperbolic discs $\Delta(z_n,\delta)$, $z_n\in\Lambda$,
  do not contain any critical points of $f$ (i.e., zeros of the derivative) by~\cite[Corollary~2]{G:2017_1}, and therefore
  $h_{z_n}' \in \H(\D)$ is zero-free for all $z_n\in\Lambda$. We compute
  \begin{equation*} 
    \frac{h_{z_n}''(z)}{h_{z_n}'(z)}
    = \frac{f''(\varphi_{z_n}(\delta z)) \, \varphi_{z_n}'(\delta z) \, \delta}{f'(\varphi_{z_n}(\delta z))}
      + \frac{\varphi_{z_n}''(\delta z) \, \delta}{\varphi_{z_n}'(\delta z)}, \quad z\in\D, \quad z_n\in\Lambda.
  \end{equation*}
 By applying~\eqref{eq:de2}, we deduce
   \begin{align*}
                                &     \sup_{z\in\D} \, \left| \frac{h_{z_n}''(z)}{h_{z_n}'(z)} \right| \\
     & \qquad \leq  \delta \, \sup_{z\in\D} \,  \frac{|f(\varphi_{z_n}(\delta z))| |A(\varphi_{z_n}(\delta z))|}{|f'(\varphi_{z_n}(\delta z))|}
       \, \frac{\big(1 - |\varphi_{z_n}(\delta z)|^2\big)^2}{1 - |\varphi_{z_n}(\delta z)|^2} \, \frac{1}{1-|\delta z|^2}
       +  \delta \, \sup_{z\in\D} \,  \frac{2 |\overline{z}_n|}{|1-\overline{z}_n \delta z|}\\
       & \qquad \leq  \frac{\delta}{1-\delta} \, \|A\|_{H^\infty_2} \sup_{z\in\Delta(z_n,\delta)} \,  \frac{|f(z)|}{|f'(z)|(1 - |z|^2)} 
       +  \frac{2\delta}{1-\delta}, \quad z_n\in\Lambda.
  \end{align*}
  Theorem~\ref{thm:factorization_h2} implies that $f=Be^g$, where $B$ is an~interpolating Blaschke product
  whose zero-sequence is $\Lambda$ and $g\in\mathcal{B}$ satisfies~\eqref{eq:iprop}. Now
  \begin{equation}  \label{eq:inf}
    \begin{split}
    \sup_{z_n\in\Lambda} \, \sup_{z\in\Delta(z_n,\delta)} \,  \frac{|f(z)|}{|f'(z)|(1 - |z|^2)}
    & =   \sup_{z\in\Delta(z_n,\delta)} \,  \frac{|B(z)|}{|B'(z)+B(z)g'(z)|(1 - |z|^2)} \\
    & \leq  \frac{\delta}{\displaystyle\inf_{z\in\Delta(z_n,\delta)} |B'(z)|(1-|z^2) - \delta \| g\|_{\mathcal{B}}}<\infty
    \end{split}
    \end{equation}
  for any sufficiently small $0<\delta<1$; the infimum in \eqref{eq:inf} is uniformly bounded away from zero for all $z_n\in\Lambda$ by
  standard estimates \cite[Lemma~2]{G:2017_1} as $\Lambda$ is uniform separated.
   We conclude
$\| h_{z_n}''/h_{z_n}' \|_{H^\infty} \lesssim \delta$
uniformly for all $z_n\in\Lambda$.
\end{proof}

The previous proof and Becker's univalence criterion \cite[Korollar 4.1]{B:1972}
imply that for any sufficiently small $0<\delta<1$, function
$h_{z_n}$ is univalent in~$\D$ for any $z_n\in\Lambda$.
Therefore the solution~$f$ is univalent in $\Delta(z_n,\delta)$ for any $z_n\in\Lambda$.


\begin{proof}[Proof of Proposition~\ref{prop:cha}]
  Assume that (i) holds, that is, $X$ is an~admissible linear subspace of $\mathcal{H}(\D)$.
  Let $w\in\H(\D)$ be locally univalent, and define $g=-(1/2)\log w' \in\H(\D)$. If $P_w \in X'$,
  then $g' = -(1/2)P_w \in X'$, and therefore $g\in X$. The admissibility implies
  \begin{equation*}
    (1/2)S_w =  -g''-(g')^2 =  R(g)'' \in X''.
  \end{equation*}
  Conversely, if $S_w\in X''$, then $R(g)'' = (1/2)S_w \in X''$. It follows that $R(g)\in X$. The admissibility implies
  $g\in X$, and therefore $-(1/2)P_w = g'\in X'$. Hence (ii) holds.
 
  Assume that (ii) holds, that is, for any locally univalent $w\in\H(\D)$, $P_w\in X'$ if and only if $S_w\in X''$.
  Let $g\in\H(\D)$, and define $$w(z) = \int_0^z e^{-2g(\zeta)}\, d\zeta , \quad z\in\D.$$ Note that $w\in\H(\D)$ is locally univalent, since
  $w'$ is zero-free. If $g\in X$, then $-(1/2)P_w = g' \in X'$. 
  The assumption implies that $ R(g)'' =  (1/2)S_w \in X''$. Therefore $R(g)\in X$.
  Conversely, if $R(g)\in X$, then $(1/2)S_w = R(g)''\in X''$. The assumption implies
  $g' = -(1/2)P_w \in X'$, and therefore $g\in X$. Hence (i) holds.
\end{proof}


\begin{proof}[Proof of Proposition~\ref{prop:aux}]
Let $X$ be an~admissible linear subspace of $\mathcal{H}(\D)$, and
let $f=e^g$ be a~zero-free solution of~\eqref{eq:de2}.
Assume that $g\in X$.
By the admissibility of $X$, we deduce $R(g)\in X$.
According to a~straight-forward computation, we get
\begin{equation*}
A =  - \frac{f''}{f} = -g'' - (g')^2 = R(g)'' \in X''.
\end{equation*}

Conversely, assume that $A \in X''$. As above,
$R(g)'' = A \in X''$,
and therefore $R(g)\in X$. The admissibility of $X$ implies $g\in X$.
\end{proof}



\begin{proof}[Proof of Theorem~\ref{thm:riccati}]
  Define the auxiliary function $h\in \H(\D)$ by
  \begin{equation*}
  h(z) = -\frac{1}{2} \int_0^z \left( B(\zeta) + \frac{C'(\zeta)}{C(\zeta)} \right) d\zeta, \quad z\in\D.
  \end{equation*}
  The assumptions imply that $h'\in X'$, and therefore $h\in X$. Recall that $g\in\H(\D)$ is a~solution
  of~\eqref{eq:riccati}, and let
\begin{equation*}
f(z) = \exp\!\left( \, \int_0^z \big( -C(\zeta) \, g'(\zeta) + h'(\zeta) \big) \, d\zeta \right),
\quad z\in\D.
\end{equation*}
By a~straight-forward computation involving~\eqref{eq:riccati}, $f$ is a~zero-free solution of $f''+af=0$,
where the coefficient $a=-f''/f$ is of the form
\begin{equation*}
a = AC + \frac{1}{2} \left( B+\frac{C'}{C} \right)' - \frac{1}{4} \left( B+\frac{C'}{C} \right)^2 = AC + R(h)''.
\end{equation*}
We deduce that $a\in X''$ by the assumption and admissibility. By Proposition~\ref{prop:aux}, we conclude that
\begin{equation*}
X' \ni \frac{f'}{f} = -C g' +h'.
\end{equation*}
The assertion $Cg'\in X'$ follows.
\end{proof}


\section{Examples of admissible linear subspaces of $\mathcal{H}(\D)$} \label{sec:examples}

It is immediate that $\H(\D)$ itself is admissible.
By \cite[Theorem~2]{Y:1977} and Proposition~\ref{prop:cha},
the Bloch space $\mathcal{B}$ is admissible.
Conversely, the space $H^\infty_\alpha$
is not admissible for any $0<\alpha<\infty$, since for $g(z)=(1-z)^{-\alpha}$ we have
$g\in H^\infty_\alpha$ while $R(g)\not\in H^\infty_\alpha$.

We turn to consider two non-trivial examples of admissible linear
subspaces of $\H(\D)$. In both cases, we conclude the admissibility by
Proposition~\ref{prop:cha}.


\subsection*{Intersection $\mathcal{B} \cap H^2$}
Suppose that $w\in \H(\D)$ is locally univalent. If $P_w \in (\mathcal{B} \cap H^2)'$,
then the standard estimate
\begin{equation} \label{eq:trivest}
  \big| S_w(z) \big|^2 \lesssim \big| P_w'(z)\big|^2 + \big| P_w(z) \big|^4, \quad z\in\D,
\end{equation}
implies $S_w\in \mathcal{B}''$.
To prove that $S_w \in (H^2)''$, we apply the
Littlewood-Paley identity \cite[Lemma~3.1, p.~236]{G:2007} and write
\begin{align*}
  \| g \|_{H^2}^2 & = |g(0)|^2 + \frac{2}{\pi} \int_{\D} |g'(z)|^2\log\frac{1}{|z|} \, dm(z)\\
  & \lesssim |g(0)|^2 + |g'(0)|^2 + \int_{\D} |g''(z)|^2 (1-|z|^2)^3\, dm(z), \quad g\in\H(\D).
\end{align*}
The assertion $S_w \in (H^2)''$ follows from the previous estimate, since
\begin{align*}
  & \int_{\D} \big| S_w(z) \big|^2 (1-|z|^2)^3 \, dm(z) \\
  & \qquad \lesssim \int_{\D} \big| P_w'(z)\big|^2 (1-|z|^2)^3 \, dm(z)
    + \int_{\D} \big| P_w(z) \big|^4 (1-|z|^2)^3 \, dm(z)\\
  & \qquad \lesssim \big(1+ \| P_w\|_{H^\infty_1}^2\big) \int_{\D} \big| P_w(z) \big|^2 (1-|z|^2) \, dm(z) < \infty.
\end{align*}
Therefore $S_w\in (\mathcal{B} \cap H^2)''$.

Conversely, let
$S_w\in (\mathcal{B} \cap H^2)''$.
By \cite[Theorem~2]{Y:1977}, we have $P_w \in \mathcal{B}'$.
When \cite[Corollary~7]{BJ:1994} is applied to $\log w'$, we deduce
\begin{align*}
\frac{2}{\pi} \int_{\D} \left| P_w(z) \right|^2 \log\frac{1}{|z|}  \, dm(z)
 & = \big\| \log w' - \log w'(0) \big\|_{H^2}^2\\
& \lesssim |w'(0)|^2 + \int_{\D} |S_w(z)|^2(1-|z|^2)^3\, dm(z) < \infty,
\end{align*}
and hence $P_w \in (H^2)'$. Therefore $P_w \in (\mathcal{B} \cap H^2)'$. We conclude that
$\mathcal{B} \cap H^2$ is admissible.



\subsection*{Analytic functions of bounded mean oscillation}

Suppose that $w\in \H(\D)$ is locally univalent. Let $P_w \in {\rm BMOA}'$.
By \eqref{eq:trivest} and \cite[Theorem~3.2]{R:2003}, we compute
\begin{align*}
  & \sup_{a\in\D} \, \int_{\D} \big| S_w(z) \big|^2 (1-|z|^2)^2 (1-|\varphi_a(z)|^2) \, dm(z) \\
  & \qquad \lesssim \sup_{a\in\D} \, \int_{\D} \big| P_w'(z)\big|^2 (1-|z|^2)^2 (1-|\varphi_a(z)|^2) \, dm(z) \\
  &  \qquad \qquad + \sup_{a\in\D} \, \int_{\D} \big| P_w(z) \big|^4 (1-|z|^2)^2 (1-|\varphi_a(z)|^2) \, dm(z)\\
  & \qquad \lesssim \big(1+ \| P_w\|_{H^\infty_1}^2\big) \, \sup_{a\in\D} \, \int_{\D} \big| P_w(z) \big|^2 (1-|\varphi_a(z)|^2) \, dm(z) < \infty,
\end{align*}
and therefore $S_w\in \BMOA''$.

Conversely, let $S_w\in \BMOA''$. By following the proof of \cite[Theorem~2(i)]{GNR:2018},
which relies heavily on \cite[Corollary~7]{BJ:1994},
we conclude that $\log w' \in \BMOA$. Therefore $P_w \in \BMOA'$. We conclude that
$\BMOA$ is admissible.


\end{document}